\documentclass[12pt]{amsart}
\usepackage{epsfig,color}

\theoremstyle{definition}

\newtheorem{theorem}{Theorem}[section]
\newtheorem{definition}[theorem]{Definition}

\newtheorem{proposition}[theorem]{Proposition}
\newtheorem{lemma}[theorem]{Lemma}
\newtheorem{remark}[theorem]{Remark}
\newtheorem{corollary}[theorem]{Corollary}
\newtheorem{example}[theorem]{Example}
\newtheorem*{acknowledgements}{Acknowledgements}

\numberwithin{equation}{section}

\newcommand{\dv}{{\text{div}}}
\newcommand{\ee}{{\text{e}}}
\newcommand{\lb}[1]{\langle#1\rangle}

\newcommand{\mf}{\mathbf}

\begin{document}

\title[Volume growth, entropy and stability for translators]{Volume growth, entropy and stability for translating solitons}

\author{Qiang Guang}

\address{Department of Mathematics, University of California, Santa Barbara, CA 93106-3080, USA}
\email{guang@math.ucsb.edu}

\thanks{}

\subjclass[2000]{53C44}

\keywords{translators, volume growth, entropy, stability, mean curvature flow}

\begin{abstract}
We study volume growth, entropy and stability for translating solitons of mean curvature flow. First, we prove that every complete properly immersed translator has at least linear volume growth. Then, by using Huisken's monotonicity formula, we compute the entropy of the grim reaper and the bowl solitons. We also give a curvature estimate for translators in $\mathbf{R}^3$ with small entropy. Finally, we estimate the spectrum of the stability operator $L$ for translators and give a rigidity result of $L$-stable translators.
\end{abstract}

\maketitle

\section{Introduction}

A smooth hypersurface $\Sigma^{n} \subset \mathbf{R}^{n+1}$ is called a \textit{translating soliton} (\textit{translator} for short) if it satisfies the equation
\begin{equation}\label{0.2}
H=-\langle y,\mathbf{n} \rangle,
\end{equation}
where $H$ is the mean curvature, $\mathbf{n}$ is the unit normal and $y\in \mathbf{R}^{n+1}$ is a constant vector. 

Translators play an important role in the study of mean curvature flow (``MCF") defined as $(\partial_t x)^\perp=-H\mathbf{n}$, where $x$ is the position vector. On one hand, every translator $\Sigma$ gives a special translating solution $\{\Sigma_t=\Sigma+ty\}_{t \in \mathbf{R}}$ to MCF. On the other hand, they arise as blow-up solutions of MCF at type \uppercase\expandafter{\romannumeral2} singularities. For instance, Huisken and Sinestrari \cite{HS} proved that at type \uppercase\expandafter{\romannumeral2} singularity of a mean convex flow (a MCF with mean convex solution), there exists a blow-up solution which is a convex translating solution. Translators have been extensively studied in recent years; see, e.g., \cite{CSS}, \cite{HR1}, \cite{IR1},  \cite{MSS}, \cite{LS2}, \cite{WXJ}, and \cite{X1}.

For simplicity, we assume that the constant vector $y=\mathbf{E}_{n+1}$, so translators satisfy the equation
\begin{equation}
H=-\lb{\mathbf{E}_{n+1},\mathbf{n}}.
\end{equation}

In $\mathbf{R}^{n+1}$, there is a unique (up to rigid motion) rotationally symmetric, strictly convex translator, denoted by $\Gamma^n$ (see \cite{AW}). For $n=1$, the translator $\Gamma^1$ is the \textit{grim reaper} and given as the graph of the function
\begin{equation}\label{0.6}
u(x)=-\log\cos x,\,\,\,\,\,\,\,x\in (-\pi/2,\pi/2).
\end{equation}
For $n\geq 2$, the translator $\Gamma^n$ is an entire graph and is usually called the \textit{bowl soliton}. We assume $\Gamma^{n}=\{(x,f(x))\in \mathbf{R}^{n+1}: x\in \mathbf{R}^{n}\}$, where $f$ is a convex function on $\mathbf{R}^{n}$ such that
\begin{equation}\label{0.7}
\dv\Big(\frac{\nabla f}{\sqrt{1 + |\nabla f|^2}}\Big)=\frac{1}{\sqrt{1+|\nabla f|^2}},
\end{equation} 
and $f(0)=\nabla f(0)=0$.

In \cite{WXJ}, Wang proved that when $n=2$, every entire convex translator must be rotationally symmetric. However, in every dimension greater than two, there exist non-rotationally symmetric, entire convex translators. Recently, Haslhofer \cite{HR1} obtained the uniqueness theorem of the bowl soliton in all dimensions under the assumptions of uniformly 2-convexity and noncollapsing condition.

\vspace{2mm}
In this paper, we study translators from three  aspects: volume growth, entropy and stability.

In the first part of the paper, we consider the volume growth of translators. Understanding the volume growth of certain geometric solitons is always a fundamental and interesting topic in the study of geometric flows, such as the volume growth of self-shrinkers for MCF and the volume growth of gradient shrinking solitons for Ricci flow. For instance, Cao and Zhou \cite{CZ1} proved that every complete noncompact gradient shrinking Ricci soliton has at most Euclidean volume growth. In \cite{DX3}, Ding and Xin showed that this is also true for complete properly immersed self-shrinkers (see also \cite{CZ}). 

Concerning the lower volume growth estimate, Munteanu and Wang \cite{MW1} proved that any complete noncompact gradient shrinking Ricci soliton has at least linear volume growth. Li and Wei \cite{LW1} proved that every complete noncompact properly immersed self-shrinker also has at least linear volume growth. For properly immersed translators, by an easy argument, an analog of Li-Wei's result can  be obtained.

\begin{theorem}
Let $\Sigma^n \subset \mathbf{R}^{n+1}$ be a complete properly immersed translator. Then for any $x\in \Sigma$, there exists a constant $C>0$ such that
\[
\text{Vol}(\Sigma\cap B_r(x))=\int_{\Sigma\cap B_r(x)} d\mu \geq Cr \,\,\,\,\, \text{for all}\,\,r\geq 1.
\]
\end{theorem}

\begin{remark}
Unlike self-shrinkers, translators do not necessarily have Euclidean volume growth; see the example constructed in \cite{N5}. In a recent work of Xin \cite{X1}, the author considered the volume growth of translators in certain  conformal metric.
\end{remark}

\vspace{2mm}
The second part of this paper is concerned with the entropy of translators. Recall that given $t_0 >0$ and $z_0 \in \mathbf{R}^{n+1}$, 
the $F$-functional $F_{z_0,t_0}$ of a hypersurface $\Sigma \subset\mathbf{R}^{n+1}$ is defined by
\begin{equation}
F_{z_0,t_0}(\Sigma)=(4\pi t_0)^{-\frac{n}{2}} \int_{\Sigma} {\textrm{e}^{-\frac{|x-z_0|^2}{4t_0}}} d\mu,
\end{equation}
and the entropy of $\Sigma$ is defined by
\begin{equation}\label{0.4}
\lambda(\Sigma)=\sup_{z_0,t_0} F_{z_0,t_0}(\Sigma),
\end{equation}
where the supremum is taken over all $t_0 >0$ and $z_0 \in \mathbf{R}^{n+1}$. Moreover, it is easy to see that $\lambda(\Sigma\times \mathbf{R})=\lambda(\Sigma)$.

Note that the entropy is invariant under dilations and rigid motions. By Huisken's monotonicity formula \cite{Hui90}, the entropy is non-increasing under MCF. Therefore, the entropy of the initial hypersurface gives a bound for the entropy of all future singularities. It was proved in \cite{CM1} that the entropy of a self-shrinker is equal to the $F$-functional $F_{0,1}$, so no supremum is needed. Colding-Ilmanen-Minicozzi-White \cite{CM3} showed that the round sphere minimizes entropy among all closed self-shrinkers. Recently, Bernstein and Wang \cite{BW} extended this result and proved that the round sphere minimizes entropy among all closed hypersurfaces up to dimension six.

A nature question is what is the entropy of the grim reaper and the bowl solitons. Using Huisken's monotonicity formula and various estimates, we obtain the following theorems.

\begin{theorem}\label{0.9}
The entropy of the grim reaper $\Gamma^{1} \subset\mathbf{R}^{2}$ is $2$.
\end{theorem}

\begin{theorem}\label{0.10}
The entropy of the rotationally symmetric convex translating soliton $\Gamma^{n} \subset\mathbf{R}^{n+1}$ ($n\geq 2)$ is equal to the entropy of the sphere $\mathbf{S}^{n-1}$, i.e.,
\[
\lambda(\Gamma^n)=\lambda(\mathbf{S}^{n-1})=n\alpha (n)\Big(\frac{n-1}{2\pi\textrm{e}}\Big)^{\frac{n-1}{2}},
\]
where $\alpha (n)$ is the volume of unit ball in $\mathbf{R}^{n}$.
\end{theorem}

If we consider weak solutions of the translator equation, more precisely, an integral rectifiable varifold, using the idea of ``boundary-less" varifold from \cite{W5} and standard geometric measure theory, then these ``boundary-less" weak solutions with small entropy are indeed smooth.
Combining this regularity result with Allard's compactness theorem, we are able to prove a curvature estimate for translators with small entropy in $\mathbf{R}^3$.

\vspace{2mm}
In the last part, we study translators from the point of view of minimal hypersurfaces. Following the notation in \cite{LS2}, for any hypersurface $\Sigma^n \subset \mathbf{R}^{n+1}$, we define the weighed functional $\mathcal{F}$ by
\[
\mathcal{F}(\Sigma)=\int_{\Sigma} \ee^{x_{n+1}}\,d\mu,
\]
where $x_{n+1}$ is the $(n+1)$-th coordinate of the position vector of $\Sigma$.

The first variation of the functional $\mathcal{F}$ gives that the critical points are translators. Therefore, translators can be viewed as minimal hypersurfaces in a conformal metric. They can also be viewed as $f$-minimal hypersurfaces with an appropriate function $f$; see \cite{IR1} where the authors study translators by using the theory of $f$-minimal hypersurfaces. By computing the second variation of the functional $\mathcal{F}$, it is natural to define the corresponding stability operator $L$ as 
\begin{equation}
L=\Delta+\lb{\mathbf{E}_{n+1},\nabla \cdot}+|A|^2.
\end{equation}

For translators, there may not be a lowest eigenvalue for the operator $L$, but we can still define the bottom of the spectrum $\mu_1$. Using the fact that $LH=0$ for translators, we show that $\mu_1$ is nonpositive if the weighed $L^2$ norm of $H$ satisfies certain growth. As a direct corollary, we obtain the following theorem.

\begin{theorem}\label{0.2.1}
For all grim reaper hyperplanes $\Gamma\times \mathbf{R}^{n-1}$, we have $\mu_1(\Gamma\times \mathbf{R}^{n-1})=0$.
\end{theorem}

In \cite{IR1}, Impera and Rimoldi proved a rigidity theorem that if an $L$-stable translator satisfies a weighted $L^2$ condition on the norm of the second fundamental form $|A|$, then it must be a hyperplane. By relaxing the condition on the weighted $L^2$ growth of $|A|$, we obtain the following rigidity result, improving Theorem A in \cite{IR1}.

\begin{theorem}\label{0.20}
Let $\Sigma^n \subset \mathbf{R}^{n+1}$ be a complete $L$-stable translator satisfying $H=-\lb{\mathbf{E}_{n+1},\mathbf{n}}$. If the norm of the second fundamental form satisfies the following weighted $L^2$ growth
\[
\int_{\Sigma\cap B_R} |A|^2 \ee^{x_{n+1}} \leq C_0 R^\alpha, \,\,\,\, \text{for any}\,\,\, R>1,
\]
where $C_0$ is a positive constant and $0\leq \alpha<2$, then $\Sigma$ is one of the following:
\begin{itemize}
\item [(1)] a hyperplane;
\item [(2)] the grim reaper $\Gamma$, when $n=1$;
\item [(3)] a grim reaper hyperplane, i.e., $\Gamma\times \mathbf{R}$, when $n=2$ and $1\leq \alpha<2$.
\end{itemize}
\end{theorem}

\begin{acknowledgements}
The author would like to thank Professor William Minicozzi for his valuable and constant support.
\end{acknowledgements}

\section{Background and Preliminaries}
In this section, we recall some background and useful identities for translators.

\subsection{Notion and conventions} 
Let $\Sigma^n \subset \mathbf{R}^{n+1}$ be a smooth hypersurface, $\Delta$ its Laplace operator, $A$ its  second fundamental form, and $H$=div$_{\Sigma}\mathbf{n}$ its  mean curvature.  If $e_i$ is an orthonormal frame for $\Sigma$, then the coefficients of the second fundamental form are given by $a_{ij}=\langle \nabla_{e_i}e_j,\mathbf{n}\rangle$.

For any hypersurface $\Sigma^n \subset \mathbf{R}^{n+1}$, the functional $\mathcal{F}$ is defined by
\[
\mathcal{F}(\Sigma)=\int_{\Sigma} \ee^{x_{n+1}}\,d\mu,
\]
where $x_{n+1}$ is the $(n+1)$-th coordinate of the position vector of $\Sigma$.

From the first variation of the functional $\mathcal{F}$, we know that the critical points are translators. The second variation formula of the functional $\mathcal{F}$ is the following.

\begin{lemma}[\cite{LS2}]\label{1.1.10}
Suppose $\Sigma^n \subset \mathbf{R}^{n+1}$ is a translator satisfying $H=-\lb{\mathbf{E}_{n+1},\mathbf{n}}$. If $\Sigma_s$ is a normal variation of $\Sigma$ with variation vector filed $\Sigma_0^{'}=f\mathbf{n}$, then 
\[
\frac{d^2}{ds^2}\Big|_{s=0}\mathcal{F}(\Sigma_s)=\int_\Sigma -\Big(\Delta f+|A|^2 f+\langle \mathbf{E}_{n+1},\nabla f\rangle\Big)f\ee^{x_{n+1}}\,d\mu.
\]
\end{lemma}

From the second variation, the stability operator $L$ is defined by
\begin{equation}\label{1.1.12}
L=\Delta+\lb{\mathbf{E}_{n+1},\nabla \cdot}+|A|^2,
\end{equation}
and the corresponding drifted operator $\mathcal{L}$ is defined by
\begin{equation}\label{1.1.14}
\mathcal{L}=\Delta+\lb{\mathbf{E}_{n+1},\nabla \cdot}=\ee^{-x_{n+1}}\text{div}(\ee^{x_{n+1}}\nabla \cdot).
\end{equation}
\begin{example}\label{1.1.13} 
We consider the stability operator $L$ for the grim reaper. The grim reaper $\Gamma$ is given by $\Gamma=(x,-\log\cos x)$, $x\in(-\pi/2,\pi/2)$. For any function $f(x)$ on $\Gamma$, by a simple computation, we have
\[
Lf=\Delta f+\lb{\mathbf{E}_{2},\nabla f}+|A|^2 f=(f_{xx}+f)\cos^2(x).
\]
In particular, if $f$ satisfies $Lf=0$, then 
\[
f(x)=a\cos x+b\sin x,
\]
for some $a,b\in \mathbf{R}$.
\end{example}

\begin{definition}
We say that a translator $\Sigma$ is $L$-stable, if for any compactly supported function $f$, we have
\[
\int_\Sigma (-fLf)\,\ee^{x_{n+1}}\,d\mu\geq 0.
\]
\end{definition}

It was proved by Shahriyari \cite{LS2} that all translating graphs are $L$-stable. Therefore, the grim reaper, the bowl solitons and hyperplanes are all $L$-stable translators.

In the next lemma, we recall some useful identities for translators.
\begin{lemma}\label{1.1.15}
If $\Sigma^n\subset \mathbf{R}^{n+1}$ is a translator satisfying $H=-\lb{\mathbf{E}_{n+1},\mathbf{n}}$, then 
\[
LA=0,
\]
\[ LH=0,
\]
\[ L|A|^2 = 2 |\nabla A|^2-|A|^4,
\]
\[ \mathcal{L}\,x_{n+1}=1.
\]
\end{lemma}
\begin{proof}
Recall that for a general hypersurface, the second fundamental form $A$ satisfies 
\begin{equation}\label{1.1.20}
\Delta A=-|A|^2A-HA^2-Hess_{H}.
\end{equation}
We fix a point $p\in \Sigma$, and choose a local orthonormal frame $e_i$ such that its
tangential covariant derivatives vanish. By the equation of translators, we obtain that
\begin{equation}\label{1.1.21}
\nabla_{e_j}\nabla_{e_i} H=a_{ij,k}\lb{\mathbf{E}_{n+1},e_k}-a_{ik}a_{jk}H.
\end{equation}
Note that $\lb{\mathbf{E}_{n+1},\nabla A}_{ij}=a_{ij,k}\lb{\mathbf{E}_{n+1},e_k}$. Combining (\ref{1.1.20}) and (\ref{1.1.21}) gives
\[
(LA)_{ij}=(\Delta A)_{ij}+|A|^2a_{ij}+\lb{\mathbf{E}_{n+1},\nabla A}_{ij}=0.
\]
This is the first identity. Taking the trace gives the second identity. For the third identity, we have
\[
L|A|^2=\mathcal{L}|A|^2+|A|^4=2\lb{A,\mathcal{L}A}+2|\nabla A|^2+|A|^4=2|\nabla A|^2-|A|^4.
\]
For the last identity, recall that in general we have $\Delta x=-H\mathbf{n}$. Hence,
\[
\mathcal{L}x_{n+1}=\Delta x_{n+1}+\lb{\mathbf{E}_{n+1},\nabla x_{n+1}}=\lb{-H\mathbf{n},\mathbf{E}_{n+1}} + |\mathbf{E}_{n+1}^T|^2=1.\]
\end{proof}

We conclude this section with the following lemma which shows that the operator $\mathcal{L}$ is self-adjoint in a weighted $L^2$ space.

\begin{lemma}\label{1.1.30}
If $\Sigma^n\subset \mathbf{R}^{n+1}$ is a translator satisfying $H=-\lb{\mathbf{E}_{n+1},\mathbf{n}}$, $u$ is a $C^2$ function with compact support, and $v$ is a $C^2$ function, then 
\[
-\int_{\Sigma} u(\mathcal{L}v)\ee^{x_{n+1}}=\int_{\Sigma} \lb{\nabla u,\nabla v}\ee^{x_{n+1}}.
\]
\end{lemma}
\begin{proof}
The lemma follows immediately from Stokes' theorem and (\ref{1.1.14}).
\end{proof}

\section{volume growth}
In this section, we consider the volume growth for translators and show that every properly immersed translator has at least linear volume growth.

Suppose $\Sigma^n \subset \mathbf{R}^{n+1}$ is a properly immersed translator. For any $x_0\in \Sigma$, let $B_r(x_0)$ be the extrinsic ball in $\mathbf{R}^{n+1}$, and denote the volume and the weighed volume of $\Sigma\cap B_r(x_0)$ by 
\[
V(r)=\text{Vol}(\Sigma\cap B_r(x_0))=\int_{\Sigma\cap B_r(x_0)} d\mu,
\]
and
\[
\tilde{V}(r)=\int_{\Sigma\cap B_r(x_0)} \ee^{x_{n+1}} d\mu.\]

We will first show that the weighed volume has at least exponential growth, and this relies on the following key ingredient.

\begin{lemma}\label{4.2.1}
If $\Sigma^n\subset \mathbf{R}^{n+1}$ is a translator satisfying $H=-\lb{\mathbf{E}_{n+1},\mathbf{n}}$, then 
\[
\Delta\ee^{x_{n+1}}=\ee^{x_{n+1}}.
\]
\end{lemma}
\begin{proof}
First, we have
\[
\Delta\ee^{x_{n+1}}=\dv_\Sigma(\ee^{x_{n+1}}\nabla_\Sigma x_{n+1})=\ee^{x_{n+1}}|\mathbf{E}_{n+1}^T|^2+\ee^{x_{n+1}}\Delta x_{n+1}.
\]
Hence, the identity follows from Lemma \ref{1.1.15} and the fact that $|\mathbf{E}_{n+1}^T|^2=1-H^2$.
\end{proof}

For simplicity, we may assume $x_0=0$. By the co-area formula, we have
\[
\tilde{V}(r)=\int_0^r \int_{\partial B_s \cap \Sigma} \ee^{x_{n+1}}\frac{1}{|\nabla_\Sigma|x||}\,ds.
\]
Note that $\nabla_\Sigma |x|^2=2x^T=2|x|\nabla_{\Sigma}|x|$, so we get
\begin{equation}\label{4.2.10}
\tilde{V}'(r)=\int_{\partial B_r \cap \Sigma} \ee^{x_{n+1}} \frac{|x|}{|x^T|}.
\end{equation}
On the other hand, by Lemma \ref{4.2.1}, we obtain that
\begin{equation}\label{4.2.11}
\begin{split}
\tilde{V}(r)=\int_{\Sigma\cap B_r} \Delta\ee^{x_{n+1}} d\mu & =\int_{\partial B_r \cap \Sigma} \big\langle \nabla_\Sigma \ee^{x_{n+1}},\frac{x^T}{|x^T|}\big\rangle \\ &=\int_{\partial B_r \cap \Sigma} \big\langle \nabla_\Sigma x_{n+1}, \frac{x^T}{|x^T|}\big\rangle \ee^{x_{n+1}}.
\end{split}
\end{equation}
Combining (\ref{4.2.10}) and (\ref{4.2.11}), we conclude that for any $r>0$,
\begin{equation}\label{4.2.13}
\tilde{V}(r)\leq \tilde{V}'(r).
\end{equation}
This implies the quantity $\tilde{V}(r)\ee^{-r}$ is monotone non-decreasing. We summarize this result in the following proposition.
\begin{proposition}\label{4.2.15}
Let $\Sigma^n \subset \mathbf{R}^{n+1}$ be a complete properly immersed translator. Then for any $x\in \Sigma$, there exists a constant $C>0$ such that
\[
\tilde{V}(r)=\int_{\Sigma\cap B_r(x)} \ee^{x_{n+1}} d\mu \geq C\ee^r, \,\,\,\,\, \text{for all}\,\,r\geq 1.\]
\end{proposition}

With the help of Proposition \ref{4.2.15}, we can now estimate the volume growth. For any $R>1$, we have
\[
\tilde{V}'(R)\leq \ee^R \int_{\partial B_R \cap \Sigma} \frac{|x|}{|x^T|} = \ee^R V'(R).
\]
Combining this with Proposition \ref{4.2.15} and (\ref{4.2.13}) gives that $V(R)$ has at least linear growth.
\begin{theorem}
Let $\Sigma^n \subset \mathbf{R}^{n+1}$ be a complete properly immersed translator. Then for any $x\in \Sigma$, there exists a constant $C>0$ such that
\[
\text{Vol}(\Sigma\cap B_r(x))=\int_{\Sigma\cap B_r(x)} d\mu \geq Cr \,\,\,\,\, \text{for all}\,\,r\geq 1.\]
\end{theorem}

\section{Entropy of the grim reaper and the bowl solitions}
The aim of this section is to estimate the entropy of the grim reaper and the bowl solitions, i.e., proving Theorem \ref{0.9} and Theorem \ref{0.10}.  Moreover, using a regularity result with Allard's compactness theorem, we will give a curvature estimate for translators with small entropy in $\mathbf{R}^3$.

\subsection{Huisken's monotonicity formula}
We state the main ingredient Huisken's monotonicity formula \cite{Hui90}. First we define the function $\Phi$ on $\mathbf{R}^{n+1} \times (-\infty,0)$ by 
\[
\Phi(x,t)=(-4\pi t)^{-\frac{n}{2}} \textrm{e}^{\frac{|x|^2}{4t}}
\]
and then set $\Phi_{z_0,\tau}(x,t)=\Phi(x-z_0,t-\tau)$. Huisken proved the following monotonicity formula for MCF:

\begin{theorem}[\cite{Hui90}]\label{2.1.2}
If $M_t$ is a solution to MCF and $u$ is a $C^2$ function, then we have
\[
\frac{d}{dt} \int_{M_t} u \Phi_{(z_0,\tau)}=-\int_{M_t} \left| H  \mathbf{n}- \frac{(x-z_0)^{\perp}
}{2(\tau-t)}\right|^2 u \Phi_{(z_0,\tau)}+\int_{M_t} (u_t-\Delta u)\Phi_{(z_0,\tau)}.
\]
\end{theorem}
When $u$ is identically one, we get 
\begin{equation}\label{2.1.3}
\frac{d}{dt} \int_{M_t}   \Phi_{(z_0,\tau)}=
-\int_{M_t} \left| H \mathbf{n}-\frac{(x-z_0)^{\perp}
}{2(\tau-t)}\right|^2 \,\Phi_{(z_0,\tau)}
 \, .
\end{equation}
The $F$-functional can be expressed as the following:
\[
F_{z_0,t_0}(\Sigma)=(4\pi t_0)^{-\frac{n}{2}} \int_\Sigma \textrm{e}^{-\frac{|x-z_0|^2}{4t_0}}=\int_\Sigma \Phi_{z_0,t_0}(x,0).
\]
Moreover, for any $z_0$ in $\mathbf{R}^{n+1}$ and $t_0>0$, if $M_t$ gives a MCF and suppose $t>s$, then by Huisken's monotonicity formula (\ref{2.1.3}), we have
\begin{equation}\label{2.1.5}
F_{z_0,t_0}(M_t)\leq F_{z_0,t_0+(t-s)}(M_s).
\end{equation}
A direct consequence of (\ref{2.1.5}) is that the entropy is non-increasing under MCF.

\vspace{2mm}
Now we apply Huisken's monotonicity formula to translating solitons.

The grim reaper or the bowl solition $\Gamma$ gives a solution of MCF  $\Gamma_t=\{(x,t+f(x))\in \mathbf{R}^{n+1}: x\in \mathbf{R}^{n}\}$\,(for $n=1$, $x \in (-\pi/2,\pi/2)$\,) for all $t\in \mathbf{R}$.

For any point  $(x_0,y_0)\in \mathbf{R}^{n+1}\, (x_0 \in \mathbf{R}^{n}, y_0 \in \mathbf{R})$ and $t_0>0$, we have the following identity:
\begin{equation}\label{2.1.6}
F_{(x_0,y_0),\,t_0}(\Gamma_t)=F_{(x_0,y_0-t),\,t_0}(\Gamma).
\end{equation}
By (\ref{2.1.5}), for any $t>s$, we get 
\[
F_{(x_0,y_0),\,t_0}(\Gamma_t)\leq F_{(x_0,y_0),\,t_0+(t-s)}(\Gamma_s).
\]
Combining this with (\ref{2.1.6}) gives
\[
F_{(x_0,y_0-t),\,t_0}(\Gamma)\leq F_{(x_0,y_0-s),\,t_0+(t-s)}(\Gamma).
\]

Therefore, we obtain the following crucial lemma.
\begin{lemma}\label{2.1.8}
For any point $\big((x_0,y_0),\,t_0\big)\in \mathbf{R}^{n+1} \times (0,\infty)$ and arbitrary $N>0$, we have  
\[
F_{(x_0,y_0),\,t_0}(\Gamma)\leq F_{(x_0,y_0+N),\,t_0+N}(\Gamma).
\]
\end{lemma}

\subsection{The entropy of the grim reaper }
Recall that the entropy is taking the supremum of the $F$-functional, so we will first consider the upper bound of the $F$-functional $F_{(x_0,y_0+N),\,t_0+N}(\Gamma)$ by taking $N \rightarrow +\infty$, and this will yield that the entropy of the grim reaper is less than or equal to 2.

For any fixed point $(x_0,y_0)\in \mathbf{R}^{2}\, $ and $t_0>0$, by Lemma \ref{2.1.8}, we may choose $N$ sufficiently large such that 
\[
F_{(x_0,y_0),\,t_0}(\Gamma)\leq F_{(x_0,N+\delta),\,N}(\Gamma),
\]
where $\delta=y_0-t_0$.

By the equation of the grim reaper (\ref{0.6}), we have 
\[
F_{(x_0,N+\delta),\,N}(\Gamma)=(4\pi N)^{-\frac{1}{2}} \int_{-\frac{\pi}{2}}^{\frac{\pi}{2}} \textrm{e}^{-\frac{(x-x_0)^2+(\log\cos x+N+\delta)^2}{4N}} \frac{1}{\cos x}\,dx.
\]
Note that
\[
 F_{(x_0,N+\delta),\,N}(\Gamma) \leq (\pi N)^{-\frac{1}{2}} \int_{0}^{\frac{\pi}{2}} \textrm{e}^{-\frac{(\log\cos x+N+\delta)^2}{4N}} \frac{1}{\cos x}\,dx.
\]
Now we define a function $g(N)$ by 
\[
g(N)=(\pi N)^{-\frac{1}{2}} \int_{0}^{\frac{\pi}{2}} \textrm{e}^{-\frac{(\log\cos x+N+\delta)^2}{4N}} \frac{1}{\cos x}\,dx.
\]
We make a change of variable and let $u=-\log\cos x$ and $t=u/N$. This gives
\[
\begin{split}
g(N^2)&=\frac{1}{N\sqrt{\pi}}\int_{0}^{\infty} \textrm{e}^{-\frac{(u-N^2-\delta)^2}{4N^2}} \frac{\textrm{e}^u}{\sqrt{\textrm{e}^{2u}-1}}\,du \\ &=\frac{1}{\sqrt{\pi}} \int_{0}^{\infty} \textrm{e}^{-\frac{(N^2+\delta)^2}{4N^2}}\,\textrm{e}^{-\frac{t^2}{4}+(\frac{3N}{2}+\frac{\delta}{2N})t} \frac{1}{\sqrt{\textrm{e}^{2Nt}-1}}\,dt.
\end{split}
\]

Next we estimate the value of $g(N^2)$ when $N$ goes to infinity by splitting it into two parts, $A$ and $B$.

For the first part,
\begin{equation}\label{2.2.8}
\begin{split}
A&=\frac{1}{\sqrt{\pi}} \int_{0}^{\frac{N}{6}} \textrm{e}^{-\frac{(N^2+\delta)^2}{4N^2}}\,\textrm{e}^{-\frac{t^2}{4}+(\frac{3N}{2}+\frac{\delta}{2N})t} \frac{1}{\sqrt{\textrm{e}^{2Nt}-1}}\,dt\\ &\leq \frac{1}{\sqrt{\pi}} \textrm{e}^{-\frac{(N^2+\delta)^2}{4N^2}+\frac{N^2}{4}+\frac{\delta}{12}} \int_{0}^{\frac{N}{6}}  \frac{1}{\sqrt{\textrm{e}^{2Nt}-1}}\,dt  \\ &\leq \frac{1}{2N\sqrt{\pi}} \textrm{e}^{-\frac{\delta^2}{4N^2}-\frac{5}{12}\delta} \int_{0}^{\frac{N^2}{3}}  \frac{1}{\sqrt{\textrm{e}^{t}-1}}\,dt \, \leq \frac{\pi}{2N\sqrt{\pi}}  \textrm{e}^{-\frac{\delta^2}{4N^2}-\frac{5}{12}\delta}. 
\end{split}
\end{equation}
Here we use the fact that 
\[
\int_{0}^{\infty}  \frac{1}{\sqrt{\textrm{e}^{t}-1}}\,dt=\pi.
\]

For the second part,
\begin{equation}\label{2.2.9}
\begin{split}
B&=\frac{1}{\sqrt{\pi}} \int_{\frac{N}{6}}^{\infty} \textrm{e}^{-\frac{(N^2+\delta)^2}{4N^2}}\,\textrm{e}^{-\frac{t^2}{4}+(\frac{3N}{2}+\frac{\delta}{2N})t} \frac{1}{\sqrt{\textrm{e}^{2Nt}-1}}\,dt \\ &\leq \frac{1}{\sqrt{\pi}}\big(1+2\,\textrm{e}^{-\frac{N^2}{{6}}}\big) \int_{\frac{N}{6}}^{\infty} \textrm{e}^{-\frac{1}{4}(t-N-\frac{\delta}{N})^2}\,dt \\ &\leq \frac{1}{\sqrt{\pi}}\big(1+2\,\textrm{e}^{-\frac{N^2}{{6}}}\big) \int_{-\frac{5N}{6}-\frac{\delta}{N}}^{\infty} \textrm{e}^{-\frac{t^2}{4}}\,dt \,\leq 2\big(1+2\,\textrm{e}^{-\frac{N^2}{{6}}}\big).
\end{split}
\end{equation}
The last inequality uses 
\[
\frac{1}{\sqrt{\pi}} \int_{-\infty}^{\infty} e^{-\frac{t^2}{4}}\,dt=2.
\]
Taking $N \rightarrow +\infty$, by Lemma \ref{2.1.8}, (\ref{2.2.8}) and (\ref{2.2.9}), we obtain the following result.

\begin{lemma}\label{2.2.10}
For any point $((x_0,y_0),\,t_0)\in \mathbf{R}^2 \times (0,\infty)$, we have 
\[
F_{(x_0,y_0),\,t_0}(\Gamma)\leq 2,\,\,\,\,
\text{i.e.},\,\,\,\,\lambda(\Gamma) \leq 2.
\]
\end{lemma}

Using the similar method as above, it is easy to prove that 
\begin{equation}\label{2.2.14}
\lim_{N\rightarrow \infty}F_{(0,N),\,N}(\Gamma)=2.
\end{equation}

Finally, combining Lemma \ref{2.2.10} and (\ref{2.2.14}), we conclude that $\lambda(\Gamma)=2$, which completes the proof of Theorem \ref{0.9}.

\subsection{The entropy of the bowl solitons}
In this subsection, we deal with the case when $n\geq 2$. The idea of the proof of Theorem \ref{0.10} is similar to the one dimensional case. The only difference is, unlike the grim reaper, we do not have explicit expression for the function $f$ in (\ref{0.7}). Therefore, we need to recall some important properties of the function $f$ (see also \cite{CSS} and \cite{WXJ}).

By the property of rotationally symmetry, we have  $f(x)=f(r)$ with $r=|x|$. The equation (\ref{0.7}) gives the following ODE
\begin{equation}\label{0.8}
f_{rr}=(1+f_{r}^2)\Big(1-\frac{(n-1)f_r}{r}\Big)
\end{equation}
with $f(0)=\lim_{r\rightarrow 0} f'(r)=0$ for $f:\mathbf{R}^{+}\rightarrow \mathbf{R}$.

\begin{proposition}\label{2.3.2} 
The function $f$ in the ODE (\ref{0.8}) satisfies the following properties:
\begin{itemize}
\item[(1)] $f'(r)<\frac{r}{n-1}$ and  $f(r) \leq \frac{r^2}{2(n-1)}$. 
\item[(2)] For any $\epsilon>0$, $f'(r)>(\frac{1}{n}-\epsilon)r$, especially, $f'(r)>\frac{r}{2n}$.
\item[(3)] For any $\epsilon>0$, there exists $r_0=r_0(\epsilon)>0$ such that $f'(r)>(1-\epsilon)\frac{r}{n-1}$ for $r\geq r_0$.
\item[(4)] For any $\epsilon>0$, there exists a constant  $M=M(\epsilon)>0$ such that $f(r)>\frac{1-\epsilon}{2(n-1)}r^2-M$.
\end{itemize}
\end{proposition}
\begin{proof}
Set $g(r)=f'(r)$. For part (1), we prove by contradiction. Set $h(r)=g(r)-\frac{r}{n-1}$ and observe that 
\[\lim_{r\to 0^+}h(r)=0\,\,\text{ and }\,\,\lim_{r\to 0^+} h'(r)=\frac{1}{n}-\frac{1}{n-1}<0.\]
Suppose that $d$ is the positive infimum of $r$ such that $h(r)=0$; that is, $h(d)=0$ and $h(r)<0$ for $r\in (0,d)$. Hence, we have \[h'(d)=g'(d)-\frac{1}{n-1}\geq 0.\]
By the equation (\ref{0.8}), we get $g'(d)=0$ since $h(d)=0$. This gives a contradiction and completes the proof of part (1). Similarly, we can prove part (2).

For part (3), we first claim that for any $R>0$, there exists $r_1\geq R$ such that \[g(r_1)>\frac{r_1}{n-1}(1-\epsilon).\]
If this is not true, then $g'(r)\geq (1+(g(r))^2)\epsilon$ for $r>R$. However, this is a contradiction, since our solution exists for all $r>0$. 

Set $k(r)=g(r)-\frac{r}{n-1}(1-\epsilon)$.
Fix $R_0>0$ to be chosen later, but depending only on $\epsilon$. The first claim implies that there exists $r_0\geq R_0$ such that $k(r_0)>0$. We next claim that $k(r)>0$ for all $r\geq r_0$. If this is not the case, then there exists $t>r_0$ such that \[k(t)=0, \,\,\,k(r)>0 \,\text{ for } r\in (r_0,t),\,\, \text{ and }\, k'(t)\leq 0.\]
Hence, we get that 
\[k'(t)=g'(t)-\frac{1-\epsilon}{n-1}=\epsilon\Big(1+\Big(\frac{1-\epsilon}{n-1}\Big)^2t^2\Big)- \frac{1-\epsilon}{n-1},\]
which is a contradiction provided that $R_0$ is chosen sufficiently large. This gives the proof of part (3). The part (4) follows directly from part (3).
\end{proof}

Next we show that the entropy of $\Gamma$ is less than or equal to the entropy of the sphere $\mf{S}^{n-1}$. Just as in the proof of Theorem \ref{0.9}, we consider the upper bound of the $F$-functional $F_{(x_0,y_0+N),\,t_0+N}(\Gamma)$ by taking $N \rightarrow +\infty$.

For any fixed point $(x_0,y_0)\in \mathbf{R}^{n+1}\, (x_0 \in \mathbf{R}^{n}, y_0 \in \mathbf{R})$ and $t_0>0$, by Lemma \ref{2.1.8}, we may choose $N$ large enough so that

\[
F_{(x_0,y_0),\,t_0}(\Gamma)\leq F_{(x_0,N+\delta),\,N}(\Gamma),
\]
where $\delta=y_0-t_0$. Then we have
\[
\begin{split}
F_{(x_0,N+\delta),\,N}(\Gamma)&=\frac{1}{(4\pi N)^{\frac{n}{2}}}\int_{\mathbf{R}^n} {\textrm{e}}^{-\frac{{|x-x_0|}^2+(f(x)-N-\delta)^2}{4N}}(1 + |\nabla f|^2)^{\frac{1}{2}}\, dx \\ &=\frac{1}{(4\pi N)^{\frac{n}{2}}}\int_0^{\infty} \int_{\partial B_r} {\textrm{e}}^{-\frac{{|x-x_0|}^2+(f(x)-N-\delta)^2}{4N}}(1 + |\nabla f|^2)^{\frac{1}{2}}dS_rdr.
\end{split}\]

Now we fix an arbitrarily  small constant $\epsilon>0$.  By Proposition \ref{2.3.2}, there exists a constant $M=M(\epsilon)>0$ such that \[f(r)>\frac{1-\epsilon}{2(n-1)}r^2-M.\] 

By the absorbing inequality, we have $2\langle x, x_0 \rangle \leq \epsilon |x|^2+\epsilon^{-1}|x_0|^2$. Therefore, we obtain that $|x-x_0|^2 \geq (1-\epsilon)|x|^2+(1-\epsilon^{-1})|x_0|^2$. This gives 
\[
\begin{split}
F_{(x_0,N+\delta),\,N}(\Gamma)& \leq \frac{1}{(4\pi N)^{\frac{n}{2}}} \int_0^{\infty} \int_{\partial B_r} (1 + |\nabla f|^2)^{\frac{1}{2}} \exp\Big\{\frac{-1}{4N}\Big((1-\epsilon)|x|^2 \\ & +(1-\epsilon^{-1})|x_0|^2+(f(x)-N-\delta)^2\Big)\Big\} \,dS_r\,dr \\ & =\frac{n\alpha(n)}{(4\pi N)^{\frac{n}{2}}}  {\textrm{e}}^{-\frac{(1-\epsilon^{-1})}{4N}|x_0|^2} \int_0^{\infty} r^{n-1}  \exp\Big\{-\frac{1}{4N}\Big((1-\epsilon)r^2 \\ & +(f(r)-N-\delta)^2\Big)\Big\}(1+(f'(r))^2)^{\frac{1}{2}}\,dr.
\end{split}
\]
Since  ${\textrm{e}}^{-\frac{(1-\epsilon^{-1})}{4N}|x_0|^2}$ goes to $1$ when $N \rightarrow +\infty$, we only need to consider the integral:
\[
g(N)=\frac{n\alpha(n)}{(4\pi N)^{\frac{n}{2}}} \int_0^{\infty} r^{n-1} {\textrm{e}}^{-\frac{(1-\epsilon)r^2+(f(r)-N-\delta)^2}{4N}}\sqrt{1+(f'(r))^2}\,dr.
\]
Note that $f(r)$ is monotone and strictly increasing, we can make a change of variable and set \[u=\frac{f(r)-N-\delta}{\sqrt{N}}.\] Assume $h$ is the inverse function of $f(r)$, then we have $h(\sqrt{N}u+N+\delta)=r$.

By Proposition \ref{2.3.2}, we get that 
\begin{equation}\label{2.3.9}
r^2\geq 2(n-1)f(r)=2(n-1)(\sqrt{N}u+N+\delta),
\end{equation}
\begin{equation}\label{2.3.10}
\sqrt{N}u+N+\delta=f(r)>\frac{1-\epsilon}{2(n-1)}r^2-M.
\end{equation}
Combining this with $f'(r)>\frac{r}{2n}$ implies that
\[
\begin{split}
g(N)&\leq \frac{n\alpha(n)}{(4\pi N)^{\frac{n}{2}}} \int_0^{\infty} {\textrm{e}}^{-\frac{(1-\epsilon)r^2+(f(r)-N-\delta)^2}{4N}}(r^{n-1}+r^{n-1}f'(r))\,dr \\ &=\frac{n\alpha(n)}{(4\pi N)^{\frac{n}{2}}} \int_{\frac{-N-\delta}{\sqrt{N}}}^\infty {\textrm{e}}^{-\frac{u^2}{4}} \,{\textrm{e}}^{-\frac{1-\epsilon}{4N}h^2}(h^{n-1}+h^{n-1}h')\sqrt{N}\,du \\ &\leq \frac{n\alpha(n)}{(4\pi N)^{\frac{n}{2}}} \int_{\frac{-N-\delta}{\sqrt{N}}}^\infty  {\textrm{e}}^{-\frac{u^2}{4}-\frac{(1-\epsilon)(n-1)}{2N}(\sqrt{N}u+N+\delta)}\big(Q^{\frac{n-1}{2}} +2n Q^{\frac{n-2}{2}}\Big)N^{\frac{1}{2}}du,
\end{split}
\]
where $h=h(\sqrt{N}u+N+\delta)$ and \[Q=\frac{2(n-1)(M+\sqrt{N}u+N+\delta)}{1-\epsilon}.\]

Next, we split the above integral into two parts, $A$ and $B$.

For part $A$, we have 
\[
A=\frac{n\alpha(n)}{(4\pi N)^{\frac{n}{2}}} \int_{\frac{-N-\delta}{\sqrt{N}}}^\infty {\textrm{e}}^{-\frac{u^2}{4}} \,{\textrm{e}}^{-\frac{(1-\epsilon)(n-1)}{2N}(\sqrt{N}u+N+\delta)}Q^{\frac{n-1}{2}}\sqrt{N}\,du.
\]
It is easy to see that
\[
\lim_{N\rightarrow +\infty} A=n\alpha (n)\Big(\frac{n-1}{2\pi (1-\epsilon)\textrm{e}^{1-\epsilon}}\Big)^{\frac{n-1}{2}}.
\]

For part $B$, we get
\[
B=\frac{n\alpha(n)}{(4\pi N)^{\frac{n}{2}}} \int_{\frac{-N-\delta}{\sqrt{N}}}^\infty {\textrm{e}}^{-\frac{u^2}{4}} \,{\textrm{e}}^{-\frac{(1-\epsilon)(n-1)}{2N}(\sqrt{N}u+N+\delta)}2nQ^{\frac{n-2}{2}}\sqrt{N}\,du,
\]
and when $N \rightarrow+\infty$, $B$ converges to 0.

Combining all the results above, we obtain that for any $\epsilon>0$,
\[
F_{(x_0,y_0),\,t_0}(\Gamma)\leq n\alpha (n)\Big(\frac{n-1}{2\pi (1-\epsilon)\textrm{e}^{1-\epsilon}}\Big)^{\frac{n-1}{2}}.
\]

Since $\epsilon$ is an arbitrarily small number, we get an upper bound of the entropy.
\begin{lemma}\label{2.3.15}
For any point $(x_0,y_0)\in \mathbf{R}^{n+1}\, (x_0 \in \mathbf{R}^{n}, y_0 \in \mathbf{R})$ and $t_0>0$, we have 
\[
F_{(x_0,y_0),\,t_0}(\Gamma)\leq n\alpha (n)\Big(\frac{n-1}{2\pi\textrm{e}}\Big)^{\frac{n-1}{2}},\,\,
\text{i.e.},\,\,\lambda(\Gamma) \leq n\alpha(n)\Big(\frac{n-1}{2\pi\textrm{e}}\Big)^{\frac{n-1}{2}}.\]
\end{lemma}

The next lemma shows that this upper bound can be achieved.
\begin{lemma}\label{2.3.17}
\begin{equation}
\lim_{N\rightarrow +\infty}F_{(0,N),\,N}(\Gamma)=n\alpha(n)\Big(\frac{n-1}{2\pi\textrm{e}}\Big)^{\frac{n-1}{2}}.
\end{equation}
\end{lemma}
\begin{proof}
By definition, we get
\begin{equation}\label{2.3.20}
\begin{split}
F_{(0,N),\,N}(\Gamma)&=(4\pi N)^{-\frac{n}{2}}\int_{\mathbf{R}^n} {\textrm{e}}^{-\frac{{|x|}^2+(f(x)-N)^2}{4N}}\sqrt{1+{|\nabla f|}^2}\, dx \\ &=\frac{n\alpha(n)}{(4\pi N)^{\frac{n}{2}}}  \int_0^{\infty} r^{n-1} {\textrm{e}}^{-\frac{r^2+(f(r)-N)^2}{4N}}\sqrt{1+(f'(r))^2}\,dr.
\end{split}
\end{equation}
Following the same method and using the same notations as in the proof of Lemma \ref{2.3.15}, by (\ref{2.3.9}) and (\ref{2.3.10}), we have 
\[
\begin{split}
F_{(0,N),\,N}(\Gamma)&=\frac{n\alpha(n)}{(4\pi N)^{\frac{n}{2}}} \int_{-\sqrt{N}}^\infty {\textrm{e}}^{-\frac{u^2}{4}} \,{\textrm{e}}^{-\frac{h^2}{4N}}\big(h^{n-1}+h^{n-1}h'\big)\sqrt{N}\,du \\& \geq \frac{n\alpha(n)}{(4\pi N)^{\frac{n}{2}}} \int_{-\sqrt{N}}^\infty \big(2(n-1)(\sqrt{N}u+N)\big)^{\frac{n-1}{2}}\sqrt{N}\exp\Big\{-\frac{u^2}{4} \\ & -\frac{2(n-1)}{4(1-\epsilon)N}\big(\sqrt{N}u+N+M\big)\Big\}\,du.
\end{split}
\]
Taking $N\rightarrow +\infty$ gives
\[
\lim_{N\rightarrow +\infty}F_{(0,N),\,N}(\Gamma)\geq n\alpha (n)\Big(\frac{n-1}{2\pi}\textrm{e}^{-\frac{1}{1-\epsilon}}\Big)^{\frac{n-1}{2}}.
\]
Since $\epsilon$ is arbitrary, by Lemma \ref{2.3.15}, the claim now easily follows.
\end{proof}

Finally, Lemma \ref{2.3.17} and Lemma \ref{2.3.15} complete the proof of Theorem \ref{0.10}.

\subsection{A curvature estimate for translators in $\mathbf{R}^3$  with small entropy}

In order to get a curvature estimate for translators in $\mathbf{R}^3$ with small entropy, we will use Allard's compactness theorem \cite{Al} for integral rectifiable varifolds with locally bounded first variation. Moreover, we will restrict to ``boundary-less varifolds" \cite{W5} where the rectifiable varifold is mod two equivalent to an integral current without boundary. Following the notation in \cite[Definition 4.1]{W5}, we say that an integral rectifiable varifold $V$ is cyclic mod 2 (boundary-less) provided $\partial [V]=0$, where $[V]$ is the rectifiable mod 2 flat chain associated to $V$. 
Note that a varifold  which consists of unions of an odd number of multiplicity-one half-planes meeting along a common line is not cyclic mod 2. Under the assumptions of small entropy and the ``boundary-less'' condition, the following regularity result holds which is implicit in section 5 in \cite{CM3} (see also section 4.1 in \cite{BW}). For convenience of the reader, we also include a proof here.

\begin{lemma}\label{2.4.1}
Let $\Sigma^2 \subset \mathbf{R}^3$ be a boundary-less (cyclic mod 2) integral rectifiable varifold with $\lambda(\Sigma)
<2$. If $\Sigma$ is a weak solution of either the self-shrinker equation or the translator equation, then $\Sigma$ is smooth.
\end{lemma}
\begin{proof}
We will analyze tangent cones to prove this lemma. By a standard result from \cite{Al} (see also section 42 in \cite{Si1}), any integral $n$-rectifiable varifold has stationary integral rectifiable tangent cones as long as the generalized mean curvature $H$ is locally in $L_p$ for some $p>n$. In our case, both equations guarantee that $H$ is locally bounded. This gives the existence of stationary tangent cones at every point. Moreover, it follows from \cite[Theorem 1.1]{W5} that those tangent cones must also be boundary-less.

Next, we will show that any such tangent cone $V$ is a multiplicity-one hyperplane. If $y\in \text{Sing}(V)\setminus \{\mathbf{0}\}$ (Sing($V$) denotes the singular set of $V$), then a dimension reduction argument \cite[Lemma 5.8]{CM3} gives that every tangent cone to $V$ at $y$ is of the form $V'\times \mathbf{R}_y$. Here $\mathbf{R}_y$ is the line in the direction $y$ and $V'$ is a one-dimensional integral stationary cone in $\mathbf{R}^2$. By the lower semi-continuity of entropy, we have $\lambda(V')\leq \lambda(V)<2$. Since any one-dimensional cone is the union of rays, the small entropy condition implies that $V'$ consists of at most three rays. Moreover, ``boundary-less'' property rules out three rays. Therefore, the only such configuration that is stationary is when there are two rays to form a multiplicity-one line and this implies that $\text{Sing}(V)\subset \{\mathbf{0}\}$. By \cite{Al2}, we conclude that the intersection of $V$ with unit sphere $\mathbf{S}^2$ is a smooth closed geodesic, i.e., a great circle, and this gives that $V$ is a multiplicity-one hyperplane.

Now combining the fact that $H$ is locally bounded, Allard's regularity theorem \cite{Al} (see also Theorem 24.2 in \cite{Si1}) gives that $\Sigma$ is a $C^{1,\beta}$ manifold for some $\beta>0$. Then elliptic theory for the self-shrinker or translator equation gives estimates on higher derivatives, and this implies that $\Sigma$ is smooth.
\end{proof}

\begin{theorem}
Let $\Sigma^2 \subset \mathbf{R}^3$ be a smooth complete translator with $\lambda(\Sigma) \leq \alpha <2$. Then there exists a constant $C=C(\alpha)>0$ such that $|A|^2\leq C$.
\end{theorem}
\begin{proof}
We will argue by contradiction. Suppose therefore that there is a sequence $\Sigma_i \subset \mathbf{R}^3$ of smooth translators with $\lambda(\Sigma_i)\leq \alpha$ and points $x_i \in \Sigma_i$ with
\begin{equation}\label{2.4.4}
|A|(x_i)>i.
\end{equation}

Then we translate $\Sigma_i$ to $\tilde{\Sigma_i}$ such that $x_i$ is the origin. Note that at any point, we have density bounds for $\tilde{\Sigma_i}$ coming from the entropy bound. Since each $\tilde{\Sigma_i}$ has bounded mean curvature, Allard's compactness theorem \cite{Al} gives a subsequence of the $\tilde{\Sigma_i}$'s that converges to an integral rectifiable varifold $\tilde{\Sigma}$ which weakly satisfies the translator equation and has $\lambda(\tilde{\Sigma})\leq \alpha <2$. The small entropy condition implies that the convergence is multiplicity one. 
As $\tilde{\Sigma_i}$ is smooth complete embedded, the Brakke flow associated to $\tilde{\Sigma_i}$ is cyclic mod 2, i.e., boundary-less, in the sense of \cite[Definition 4.1]{W5}. Therefore, by \cite[Theorem 4.2]{W5}, the limit $\tilde{\Sigma}$ is also boundary-less. Thus, Lemma \ref{2.4.1} gives that $\tilde{\Sigma}$ is smooth. Finally, Allard's regularity theorem \cite{Al} implies that the convergence is also smooth, which contradicts (\ref{2.4.4}).
\end{proof}
\begin{remark}
A similar result holds for $L$-stable translators in $\mathbf{R}^3$. More precisely, there exists a constant $C>0$ such that the curvature of any $L$-stable translator in $\mathbf{R}^3$ is bounded by $C$, i.e., $|A|\leq C$ (see \cite{GZ4}). The translating graph case of this result was proved by Shahriyari \cite[Theorem 3.2]{LS2}.
\end{remark}

\section{$L$-stability and Rigidity results}
In this section, we first consider the spectrum of the operator $L$ for translators and compute the bottom of the spectrum for grim reaper hyperplanes. Then, by using a cut off argument and a uniqueness lemma, we give a rigidity result for translators in terms of the weighed $L^2$ norm of the second fundamental form.
\subsection{The spectrum of $L$ for translators}Let $\Sigma^n \subset \mathbf{R}^{n+1}$ be a translator.  Note that all translators are noncompact, so there may not be a lowest eigenvalue for the operator $L$. However, we can still define the bottom of the spectrum $\mu_1$ by 
\[
\mu_1=\inf_f \frac{\int_\Sigma \big(|\nabla f|^2-|A|^2f^2\big)\ee^{x_{n+1}}}{\int_\Sigma f^2 \ee^{x_{n+1}}}=\inf_f \frac{-\int_\Sigma f(Lf)\ee^{x_{n+1}}}{\int_\Sigma f^2 \ee^{x_{n+1}}},
\]
where the infimum is taken over all smooth functions with compact support. Note that standard density arguments imply that we get the same $\mu_1$ if we take the infimum over 
Lipschitz functions with compact support. It is possible that $\mu_1=-\infty$ since all translators are noncompact.

By the definition of $\mu_1$ and the minimal surface theory (see \cite{FS1}; cf.\cite[Lemma 9.25]{CM1}), we have the following characterization of $L$-stability for translators.
\begin{lemma}\label{3.1.5}
If $\Sigma^n \subset \mathbf{R}^{n+1}$ is a translator, then the following are equivalent:
\begin{itemize}
\item [(1)] $\Sigma$ is $L$-stable;
\item [(2)] $\mu_1(\Sigma) \geq 0$;
\item [(3)] there exists a positive function $u$ satisfying $Lu=0$.
\end{itemize}
\end{lemma}

Next, we show that if the mean curvature $H$ of a translator has at most quadratic weighted $L^2$ growth, then the bottom of the spectrum $\mu_1$ is nonpostive.

\begin{lemma}\label{3.1.10}
Let $\Sigma^n \subset \mathbf{R}^{n+1}$ be a complete properly immersed translator. If the mean curvature $H$ satisfies the following weighted $L^2$ growth
\begin{equation}\label{3.1.11}
\int_{\Sigma\cap B_R} H^2 \ee^{x_{n+1}} \leq C_0 R^\alpha, \,\,\,\, \text{for any}\,\,\, R>1,
\end{equation}
where $C_0$ is a positive constant and $0\leq \alpha<2$,
then we have $\mu_1(\Sigma)\leq 0$.
\end{lemma}

\begin{proof}
Given any fixed $\delta>0$, if we can construct a compactly supported function $u$ such that
\begin{equation}\label{3.1.12}
-\int_\Sigma u(Lu)\ee^{x_{n+1}} <\delta\int_\Sigma u^2\ee^{x_{n+1}},
\end{equation}
then this implies $\mu_1\leq 0$.

We will use the mean curvature $H$ and the fact that $LH=0$ to construct our test function. Suppose $\eta$ is a function with compact support and let $u=\eta H$. Then
\[
L(\eta H)= \eta LH+H\mathcal{L}\eta +2\lb{\nabla \eta, \nabla H}=H\mathcal{L}\eta +2\lb{\nabla \eta, \nabla H}.
\]
It follows that
\begin{equation}\label{3.1.15}
-\int_\Sigma \eta HL(\eta H)\ee^{x_{n+1}}=-\int_\Sigma \Big[ \eta H^2 \mathcal{L}\eta +2\eta H \lb{\nabla \eta, \nabla H}\Big]\ee^{x_{n+1}}. 
\end{equation}
Applying Lemma \ref{1.1.30} gives 
\begin{equation}\label{3.1.16}
-\int_\Sigma H^2 \mathcal{L}(\eta^2) \ee^{x_{n+1}}=\int_\Sigma \lb{\nabla H^2, \nabla \eta^2}\ee^{x_{n+1}}.
\end{equation}
Note that $\mathcal{L}\eta^2=2\eta \mathcal{L}\eta + 2|\nabla \eta|^2$. Combining this with  (\ref{3.1.15}) and (\ref{3.1.16}) yields that
\begin{equation}\label{3.1.18}
-\int_\Sigma \eta HL(\eta H)\ee^{x_{n+1}}=\int_\Sigma H^2|\nabla \eta|^2\ee^{x_{n+1}}.
\end{equation} 
If we choose $\eta$ to be one on $B_R$ and cut off linearly to zero on $B_{2R}\backslash B_R$, then (\ref{3.1.18}) gives
\begin{equation}\label{3.1.20}
\begin{split}
-\int_\Sigma \big(\eta HL(\eta H)+\delta \eta^2H^2\big)\ee^{x_{n+1}} \leq & \frac{1}{R^2}\int_{\Sigma \cap B_{2R}}H^2 \ee^{x_{n+1}} \\ &-\delta \int_{\Sigma \cap B_R} H^2 \ee^{x_{n+1}}.
\end{split}
\end{equation}
By the assumption (\ref{3.1.11}), the right-hand side of (\ref{3.1.20}) must be negative for sufficiently large $R$. Therefore, when $R$ is large, the function $u=\eta H$ satisfies (\ref{3.1.12}), and this completes the proof.
\end{proof}

Lemma \ref{3.1.5} and Lemma \ref{3.1.10} imply that any complete $L$-stable translator satisfying (\ref{3.1.11}) has $\mu_1=0$. In particular, the following result, i.e., Theorem \ref{0.2.1} holds.
\begin{corollary}\label{3.1.25}
For all grim reaper hyperplanes $\Gamma\times \mathbf{R}^{n-1}$, we have $\mu_1(\Gamma\times \mathbf{R}^{n-1})=0$.
\end{corollary}
\begin{remark}
If $n\geq 3$, Corollary \ref{3.1.25} does not follow directly from Lemma \ref{3.1.10}, but we can slightly modify the argument in the proof of Lemma \ref{3.1.10} to get the result.
\end{remark}

Using a similar argument as in the proof of  \cite[Lemma 9.25]{CM1}, we have the following characterization of $\mu_1$ for translators.
\begin{lemma}\label{3.1.35}
If $\Sigma^n \subset \mathbf{R}^{n+1}$ is a complete properly immersed translator and $\mu_1(\Sigma)\neq-\infty$, then there is a positive function $u$ on $\Sigma$ such that $Lu=-\mu_1 u$. 
\end{lemma}

\subsection{Uniqueness lemma and rigidity results}
The proof of Theorem \ref{0.20} relies on the following standard uniqueness lemma for general hypersurfaces. 
\begin{lemma}\label{3.2.5}
Let $\Sigma^n \subset \mathbf{R}^{n+1}$ be a smooth complete hypersurface. If $g>0$ and $h$ are two functions on $\Sigma$ which satisfy
\begin{equation}\label{3.2.6}
\Delta g+\lb{\nabla f, \nabla g}+Vg=0,
\end{equation}
and
\begin{equation}\label{3.2.7}
\Delta h+\lb{\nabla f, \nabla h}+Vh=0,
\end{equation}
where $f$ and $V$ are smooth functions on $\Sigma$. If $\Sigma$ is closed or $h$ satisfies the weighted $L^2$ growth
\begin{equation}\label{3.2.8}
\int_{\Sigma\cap B_R} h^2 \ee^f \leq C_0R^{\alpha},\,\,\,\, \text{for any}\,\,\,  R>1,
\end{equation}
where $C_0$ is a positive constant and $0\leq \alpha<2$, then $h=Cg$ for some constant $C$.
\end{lemma}
\begin{proof}
For convenience, we let $\mathcal{L}=\Delta+\lb{\nabla f,\nabla \cdot}$. Thus, $\mathcal{L}g+Vg=0$ and $\mathcal{L}h+Vh=0$.

Let $w=\frac{h}{g}$. Then by (\ref{3.2.6}) and (\ref{3.2.7}), we have
\begin{equation}\label{3.2.10}
\mathcal{L}w=\frac{g\mathcal{L}h-h\mathcal{L}g}{g^2}-2\lb{\nabla w,\frac{\nabla g}{g}}= -2\lb{\nabla w,\frac{\nabla g}{g}}.
\end{equation}
We define a drifted operator $\mathcal{L}_g$ by
\begin{equation}\label{3.2.12}
\mathcal{L}_g=\frac{\ee^{-f}}{g^2}\text{div}(g^2\ee^f\,\nabla \cdot)=\mathcal{L}+2\lb{\nabla \cdot,\frac{\nabla g}{g}}.
\end{equation}
Then (\ref{3.2.10}) gives 
\begin{equation}\label{3.2.14}
\mathcal{L}_g w^2=2w\mathcal{L}_g w+2|\nabla w|^2=2|\nabla w|^2 \geq 0.
\end{equation}
Now, if $\Sigma$ is closed, then integrating (\ref{3.2.14}) finishes the proof. If $\Sigma$ is noncompact, we choose a cut off function $\phi$. By (\ref{3.2.12}) and (\ref{3.2.14})
\begin{equation}\label{3.2.15}
\begin{split}
\frac{\ee^{-f}}{g^2}\text{div}(\phi^2g^2\ee^f\,\nabla w^2) & =\lb{\nabla \phi^2,\nabla w^2}+\phi^2\mathcal{L}_gw^2 \\ & =\lb{\nabla \phi^2,\nabla w^2}+2\phi^2|\nabla w|^2.
\end{split}
\end{equation}
Using the Stokes' theorem, (\ref{3.2.15}) gives
\begin{equation}\label{3.2.16}
0=\int_\Sigma \big[\lb{\nabla \phi^2,\nabla w^2}+\phi^2\mathcal{L}_gw^2\big]g^2 \ee^f.
\end{equation}
Applying the absorbing inequality, (\ref{3.2.16}) gives
\[
0\geq \int_\Sigma \Big[-\phi^2|\nabla w|^2-4w^2|\nabla \phi|^2+2\phi^2|\nabla w|^2\Big]g^2\ee^f.
\]
This is equivalent to 
\[
\int_\Sigma \phi^2|\nabla w|^2 g^2 \ee^f \leq 4\int_\Sigma w^2 |\nabla \phi|^2 g^2 \ee^f.
\]
If we choose $\phi$ to be identically one on $B_R$ and cuts off linearly to zero from $\partial B_R$ to $\partial B_{2R}$, then $|\nabla \phi| \leq 1/R$. Combining the weighted growth of $h$, i.e.,   (\ref{3.2.8}) and taking $R\rightarrow \infty$, we conclude that $|\nabla w|=0$. This completes the proof.
\end{proof}

\begin{remark}\label{3.2.20}
We can slightly modify the proof of Lemma \ref{3.2.5} to show that it still holds if we assume $\Delta h+\lb{\nabla f, \nabla h}+Vh\geq 0$ and $h\geq 0$, which is a special case of \cite[Theorem 8]{R3}.
\end{remark}

As an application of Lemma \ref{3.2.5} to $L$-stable translators, we have the following lemma.
\begin{lemma}\label{3.2.25}
Let $\Sigma^n \subset \mathbf{R}^{n+1}$ be a complete $L$-stable translator. If the mean curvature $H$ satisfies the weighed $L^2$ growth
\begin{equation}
\int_{\Sigma\cap B_R} H^2 \ee^{x_{n+1}} \leq C_0 R^\alpha, \,\,\,\, \text{for any}\,\,\, R>1,
\end{equation}
where $C_0$ is a positive constant and $0\leq \alpha<2$, then $H\equiv0$ or $H$ does not change sign.
\begin{proof}
This follows from Lemma \ref{3.1.5}, Lemma \ref{3.2.5} and the fact that $LH=0$.
\end{proof}
\end{lemma}

We are now ready to prove Theorem \ref{0.20}.

\begin{proof}[Proof of Theorem \ref{0.20}]
First by Lemma \ref{3.2.25}, we conclude that $H\equiv 0$ or $H$ does not change sign. If $H\equiv 0$, then $\Sigma$ is just a hyperplane, which is the first case. Next, we assume $H$ does not change sign and $H>0$. Then $|A|$ does not vanish.

By Lemma \ref{1.1.15}, we have
\[
L|A|=\frac{|\nabla A|^2-|\nabla |A||^2}{|A|} \geq 0.
\]
Since $LH=0$ and $L|A|\geq 0$, applying Lemma \ref{3.2.5} with $g=H$ and $h=|A|$, by Remark \ref{3.2.20}, we conclude that there exists a constant $C$ such that 
\[
|A|=CH.
\]
It follows that $\Sigma$ is the grim reaper or a grim reaper hyperplane $\Gamma\times \mathbf{R}^{n-1}$ (see Theorem B in \cite{MSS} or \cite{IR1}).

Note that if $n\geq 3$, the norm of the second fundamental form $|A|$ of a grim reaper hyperplane $\Gamma\times \mathbf{R}^{n-1}$ has at least quadratic weighted $L^2$ growth. Therefore, by our assumption, we conclude that $n\leq 2$. If $n=1$, then $\Sigma$ is the grim reaper $\Gamma$, which is the second case. If $n=2$, then $\Sigma$ is a grim reaper hyperplane $\Gamma \times \mathbf{R}$ and $1\leq \alpha<2$, which is the last case.
\end{proof}

\end{document}